\documentclass{article}
\usepackage {amssymb}
\usepackage {amsmath}
\usepackage {CJK}
\usepackage {amsmath}
\usepackage {amsthm}
\usepackage{graphicx}
\usepackage {amscd}
\usepackage[colorlinks, citecolor=red]{hyperref}
\newtheorem{thm}{Theorem}[section]

\newtheorem{rem}{Remark}

\newtheorem{lem}{Lemma}

\begin{document}
\title{\textbf{On rotationally symmetric K\"{a}hler-Ricci solitons}}
\author{Chi Li}
\date{}
\maketitle
\begin{abstract}
\noindent ABSTRACT: In this note, using Calabi's method, we
construct rotationally symmetric K\"{a}hler-Ricci solitons on the
total space of direct sum of fixed hermitian line bundle and its
projective compactification, where the curvature of hermitian line
bundle is K\"{a}hler-Einstein. These examples generalize the
construction of Koiso, Cao and Feldman-Ilmanen-Knopf.
\end{abstract}

\begin{section}{A little motivation}
In \cite{FIK}, the authors constructed some examples of gradient
K\"{a}hler-Ricci soliton. Among them is the shrinking soliton on the
$Bl_0\mathbb{C}^m$. They also glue this to an expanding soliton on
$\mathbb{C}^m$ to extend the Ricci flow across singular time.

Recently, La Nave and Tian \cite{NT} studied the formation of
singularity along K\"{a}hler-Ricci flow by symplectic quotient. The
idea is explained by the following example.

Let $\mathbb{C}^*$ act on $\mathbb{C}^{m+n}$ by
\[t\cdot (x_1,\cdots,x_m,y_1,\cdots,y_n)=(t\,x_1,\cdots, t\,x_m,t^{-1}\, y_1,\cdots,t^{-1}\, y_n)\]
$S^1\subset \mathbb{C}^*$ preserves the standard K\"{a}hler
structure on $\mathbb{C}^{m+n}$: \[\omega=\sqrt{-1}(\sum_{i=1}^m
dx_i\wedge d\bar{x}_i+\sum_{\alpha=1}^{n}dy_\alpha\wedge
d\bar{y}_\alpha)\] Let $z=(x,y)=(x_1,\cdots,x_m,y_1,\cdots,y_n)$.
The momentum map of this Hamiltonian action is
\[m(z)=\sum_{i=1}^m |x_i|^2-\sum_{\alpha=1}^n |y_\alpha|^2=|x|^2-|y|^2\]
The topology of symplectic quotient $X_a=m^{-1}(a)/S^1$ changes as
$a$ across 0.

Let $\mathcal{O}_{\mathbb{P}^N}(-1)$ be the tautological line bundle
on the complex projective space $\mathbb{P}^N$. We will use
$Y_{N,R}$ to represent the total space of holomorphic vector bundle
$(\mathcal{O}_{\mathbb{P}^N}(-1)^{\oplus R}\rightarrow
\mathbb{P}^N)$.
\begin{enumerate}
\item (a$>$0) $\forall z=(x,y)\in X_a$, $m(z)=|x|^2-|y|^2=a>0$, so $x\neq
0$.

$X_a\simeq Y_{m-1,n}\simeq
\{\mathbb{C}^{m+n}-\{x=0\}\}/\mathbb{C}^*$. The isomorphism is given
by
\[ (x_1,\cdots,x_m,y_1,\cdots, y_n)\mapsto ([x_1,\cdots,x_m], y_1\cdot x,\cdots, y_n\cdot
x)\] There is an induced K\"{a}hler metric on $X_a$. Choose a
coordinate chart $u_1=\frac{x_2}{x_1},\cdots,
u_{m-1}=\frac{x_m}{x_1}, \xi_1=x_1y_1,\cdots, \xi_n=x_1y_n$. The
$C^*$ action is then trivialized to: $(x_1,u,\xi)\mapsto (t
x_1,u,\xi)$. The K\"{a}hler potential can be obtained by some
Legendre transformation (see \cite{BG}). Specifically, the potential
for the standard flat K\"{a}hler metric on
$\{\mathbb{C}^{m+n}-\{x=0\}\}$ is
\[
\phi=|x|^2+|y|^2=|x_1|^2(1+|u|^2)+\frac{|\xi|^2}{|x_1|^2}=e^{r_1}(1+|u|^2)+e^{-r_1}|\xi|^2
\]
where $r_1=\log|x_1|^2$. $\phi$ is a convex function of $r_1$.
$a=\frac{\partial\phi}{\partial r_1}$ is the momentum map of the
$S^1$ action. In the induced coordinate chart $(u,\xi)$, the
K\"{a}hler potential of the induced metric on the symplectic
quotient is the Legendre transform of $\phi$ with respect to $r_1$:
\begin{equation}\label{phia}
\Phi_a=a\log(1+|u|^2)+\sqrt{a^2+4(1+|u|^2)|\xi|^2}-a\log(a+\sqrt{a^2+4(1+|u|^2|)|\xi|^2})+(\log2)\,a
\end{equation}

\item (a$<$0) By symmetry, $X_a\simeq Y_{n-1,m}\simeq\{\mathbb{C}^{m+n}-\{y=0\}\}/\mathbb{C}^*$. Choose a
coordinate chart $v_1=\frac{y_2}{y_1},\cdots,
v_{n-1}=\frac{y_n}{y_1}, \eta_1=y_1x_1,\cdots, \eta_m=y_1x_m$. The
K\"{a}hler potential has the same expression as (\ref{phia}) but
replacing $a$ by $-a$, $u$ by $v$, and $\xi$ by $\eta$.
\item (a$=$0) $X_a\cong \mbox{ affine cone over the Segre embedding
of } \mathbb{P}^{m-1}\times\mathbb{P}^{n-1}\hookrightarrow
\mathbb{P}^{mn-1}$:
\[(x_1,\cdots,x_m,y_1,\cdots,y_n)\mapsto \{x_iy_\alpha\}\]
Away from the vertex of the affine cone, choose a coordinate chart:
$u_1=\frac{x_2}{x_1},\cdots, u_{m-1}=\frac{x_m}{x_1},
v_1=\frac{y_2}{y_1},\cdots, v_{n-1}=\frac{y_n}{y_1}$,
$\zeta=x_1y_1$. The K\"{a}hler potential is given by
\[\Phi_0=2\sqrt{(1+|u|^2)(1+|v|^2)|\zeta|^2}\]
Note that $\Phi_0$ is obtained from $\Phi_a$ by coordinate change
$\xi_1=\zeta, \xi_2=v_1\zeta,\cdots,\xi_n=v_{n-1}\zeta$, and let $a$
tend to 0.
\end{enumerate}
This is a simple example of flip when $m\neq n$, or flop when $m=n$,
in the setting of symplectic geometry. $X_{<0}$ is obtained from
$X_{>0}$ by first blowing up the zero section $\mathbb{P}^{m-1}$,
and then blowing down the exceptional divisor
$E\cong\mathbb{P}^{m-1}\times\mathbb{P}^{n-1}$ to
$\mathbb{P}^{n-1}$. Note that when $n=1$, this process is just
blow-down of exceptional divisor in $Bl_0\mathbb{C}^m$.

\begin{figure}
\centering
\includegraphics{flip.1}
\end{figure}

One hopes to have a K\"{a}hler metric on a larger manifold
$\mathcal{M}$ such that induced K\"{a}hler metrics on symplectic
quotient would satisfy the K\"{a}hler-Ricci flow equation as the
image of momentum varies. See \cite{NT} for details.

Our goal to construct a K\"{a}hler-Ricci soliton on $Y=Y_{m-1,n}$
and its projective compactification, and this generalizes
constructions of \cite{Koi}, \cite{Cao} \cite{FIK}. The construction
follows these previous constructions closely, but we need to modify
them to fit our setting. The higher dimensional analogs have the new
phenomenon of contracting higher codimension subvariety to highly
singular point. To continue the flow, surgery are needed. The
surgeries in these cases should be the naturally appearing flips.

The organization of this note is as follows. In section 2, we put
the construction in a more general setting where the base manifold
is K\"{a}hler-Einstein, and state the main results: Theorem
\ref{KRsoliton} and Theorem \ref{cptshrink}. In section 3, by the
rotational symmetry, we reduce the K\"{a}hler-Ricci soliton equation
to an ODE. In section 4, we analyze the condition in order for the
general solution of the ODE to give a smooth K\"{a}hler metric near
zero section. In section 5.1, we get the condition for the metric to
be complete near infinity. In section 5.2, we prove theorem
\ref{KRsoliton}, i.e. construct K\"{a}hler-Ricci solitons in the
noncompact complete K\"{a}hler manifold and study its behavior as
time approaches the singular time. Finally in section 6, we prove
theorem \ref{cptshrink} by constructing the compact shrinking
soliton on projective compactification.

The author thanks: Professor Gang Tian for constant encouragement;
Professor Jian Song and Yuan Yuan for helpful discussions; Professor
Xiaohua Zhu for suggesting the author to study the compact shrinking
case; Zhou Zhang for his interest in these examples. He is
particularly grateful to Professor Jian Song for many insightful
suggestions and great encouragement.
\end{section}
\begin{section}{General setup and the result}
Let $M$ be a K\"{a}hler manifold of \textbf{dimension d}.
K\"{a}hler-Ricci soliton on $M$ is a K\"{a}hler metric $\omega$
satisfying the equation
\begin{equation}\label{soleq}
Ric(\omega)=\lambda\omega+L_V\omega
\end{equation}
where $V$ is a holomorphic vector field. The K\"{a}hler-Ricci
soliton is called gradient if $V=\nabla f$ for some potential
function $f$. If $\sigma(t)$ is the 1-parameter family of
automorphisms generated by $V$, then
\begin{equation}\label{pullback}
\omega(t)=(1-\lambda t)\sigma\left(-\frac{1}{\lambda}\log(1-\lambda
t)\right)^*\omega
\end{equation}
is a solution of K\"{a}hler-Ricci flow equation:
\[
\frac{\partial\omega(t)}{\partial t}=-Ric(\omega(t))
\]
We will construct gradient K\"{a}hler-Ricci solitons on the total
space of special vector bundle $ L^{\oplus n}\rightarrow M $ and its
projective compactification $\mathbb{P}(\mathbb{C}\oplus L^{\oplus
n})=\mathbb{P}(L^{-1}\oplus\mathbb{C}^{\oplus n})$. Here $M$ is a
K\"{a}hler-Einstein manifold:
\[
Ric(\omega_M)=\tau\omega_M
\]
$L$ has an Hermitian metric $h$, such that
\[
c_1(L,h)=-\sqrt{-1}\partial\bar{\partial}\log h=-\epsilon\omega_M
\]
In the following, we always consider the case $\epsilon\ge 0$.

We consider the K\"{a}hler metric of the form considered by Calabi
\cite{Ca}:
\begin{equation}\label{kafm}
\omega=\pi^*\omega_M+\partial\bar{\partial}P(s)
\end{equation}
Here $s$ is the norm square of vectors in $L$. Under local
trivialization of holomorphic local section $e_L$,
\[
s(\xi e_L)=a(z)|\xi|^2,\; \xi=(\xi^1,\cdots,\xi^n)
\]
$P$ is a smooth function of $s$ we are seeking for.

Using the form (\ref{kafm}), we can determine $\lambda$ immediately.
Let $M$ be the zero section. By adjoint formula,
\[
-K_Y|_{M}=-K_{M}+\wedge^nN_M|_{M}=-K_{M}+nL
\]
Note that $\omega|_{M}=\omega_{M}$, so by restricting both sides of
(\ref{soleq}) to $M$, and then taking cohomology, we see that
\begin{equation}\label{valam}
\tau[\omega_M]-n\epsilon[\omega_M]=c_1(Y)|_M=\lambda[\omega_M]
\end{equation}
So $\lambda=\tau-n\epsilon$.
\begin{rem}
If we rescale the K\"{a}hler-Einstein metric:
$\omega_M\rightarrow\kappa\omega_M$, then
$\tau\rightarrow\tau/\kappa$, $\epsilon\rightarrow\epsilon/\kappa$,
$\lambda\rightarrow\lambda/\kappa$.
\end{rem}
The main theorem is
\begin{thm}\label{KRsoliton}
On the total space of $L^{\oplus n}$, there exist rotationally
symmetric solitons of types depending on the sign of
$\lambda=\tau-n\epsilon$. If $\lambda>0$, there exists a unique
shrinking soliton. If $\lambda=0$, there exists a family of steady
solitons. If $\lambda<0$, there exists a family of expanding
solitons. (The solitons are rotationally symmetric in the sense that
it's of the form of \eqref{kafm})
\end{thm}
\begin{rem}
If we take $M=\mathbb{P}^{m-1}$, $L=\mathcal{O}(-1)$,
$\omega_M=\omega_{FS}$, then $\tau=m$, $\epsilon=1$. Then we get to
the situation in section 1. So depending on the sign of
$\lambda=m-n$, there exist either a unique rotationally symmetric
shrinking KR soliton when $m>n$, or a family of rotationally
symmetric steady KR solitons when m=n, or a family of rotationally
symmetric expanding KR solitons when $m<n$.
\end{rem}
We also have the compact shrinking soliton:
\begin{thm}\label{cptshrink}
Using the above notation, assume $\lambda=\tau-n\epsilon>0$, then on
the space $\mathbb{P}(\mathbb{C}\oplus L^{\oplus
n})=\mathbb{P}(L^{-1}\oplus\mathbb{C}^{\oplus n})$, there exists a
unique shrinking K\"{a}hler-Ricci soliton.
\end{thm}
\end{section}

\begin{section}{Reduction to ODE}
The construction of solitons is straightforward by reducing the
soliton equation to an ODE.

First, in local coordinates, \eqref{kafm} is expressed as
\begin{equation}\label{omega}
\omega=(1+\epsilon
P_ss)\omega_M+a(P_s\delta_{\alpha\beta}+P_{ss}a\overline{\xi^\alpha}\xi^\beta)
\nabla\xi^\alpha\wedge\overline{\nabla\xi^\beta}
\end{equation}
Here
\[
\nabla\xi^\alpha=d\xi^\alpha+a^{-1}\partial a\,\xi^\alpha
\]
Note that $\{dz^i,\nabla\xi^\alpha\}$ are dual to the basis
consisting of horizontal and vertical vectors:
\[\nabla_{z^i}=\frac{\partial}{\partial z^i}-a^{-1}\frac{\partial a}{\partial
z^i}\sum_\alpha\xi^\alpha\frac{\partial}{\partial\xi^\alpha},\;\;\frac{\partial}{\partial\xi^\alpha}\]

$\omega$ is positive if and only if
\begin{equation}\label{poscond}
1+\epsilon P_ss>0,\, P_s>0, \mbox{ and } P_s+P_{ss}s>0
\end{equation}
\[
\omega^{d+n}=(1+\epsilon
P_ss)^{d}\omega_M^da^nP_s^{n-1}(P_s+P_{ss}s)\prod_{\alpha=1}^n
d\xi^\alpha\wedge d\bar{\xi}^\alpha
\]
Since we assume
$Ric(\omega_M)=\tau\omega_M=(\lambda+n\epsilon)\omega_M$,
\[
\partial\bar{\partial}\log\det\omega^{d+n}+\lambda(\omega_M+\partial\bar{\partial}P)=\partial\bar{\partial}
\left[d\cdot\log(1+\epsilon P_ss)+(n-1)\log P_s+\log
((P_ss)_s)+(\tau-n\epsilon) P\right]
\]
Let $r=\log s$, then $\partial_r=s\partial_s$. Define
\begin{eqnarray}\label{Qfunc}
Q&:=&d\cdot\log(1+\epsilon P_ss)+(n-1)\log P_s+\log
((P_ss)_s)+(\tau-n\epsilon) P\nonumber\\
&=&d\cdot\log(1+\epsilon P_r)+(n-1)\log P_r+\log
P_{rr}-nr+(\tau-n\epsilon)P
\end{eqnarray}
To construct a gradient K\"{a}hler-Ricci soliton \eqref{soleq}, it
is sufficient to require that $Q(t)$ is a potential function for the
holomorphic vector field $-V$. Notice that, for the radial
holomorphic vector field:
\begin{equation}
V_{rad}=\sum_{\alpha=1}^n\xi^\alpha\frac{\partial}{\partial
\xi^\alpha}
\end{equation}
\[
i_{V_{rad}}\omega=(P_s+P_{ss}s)a\sum_\beta\xi^\beta
\overline{\nabla\xi^\beta}=(P_ss)_s\bar{\partial}s
\]
Now
\[
-i_V\omega=\bar{\partial}Q(s)=Q_s\bar{\partial}s=\frac{Q_s}{(P_ss)_s}i_{V_{rad}}\omega
\]
which means $-V=\frac{Q_s}{(P_ss)_s}V_{rad}$, so
$\frac{Q_s}{(P_ss)_s}$ is a holomorphic function. Since
$s=a(z)|\xi|^2$ is not holomorphic, $\frac{Q_s}{(P_ss)_s}$ has to be
a constant $\mu$. We assume $\mu\neq0$, since $V\neq 0$. So we get
the equation: $Q_s=\mu (P_ss)_s$. Multiplying by s on both sides,
this is equivalent to
\begin{equation}\label{soleq2}
Q_r=\mu P_{rr}
\end{equation}
\begin{rem}
Note that, if $\mu=0$, then we go back to Calabi's construction of
K\"{a}hler-Einstein metrics in \cite{Ca}.
\end{rem}
Define $\phi(r)=P_r(r)$ and substitute \eqref{Qfunc} into
\eqref{soleq2}, then we get
\begin{equation}\label{soleq3}
d\frac{\epsilon\phi_r}{1+\epsilon\phi}+(n-1)\frac{\phi_r}{\phi}+\frac{\phi_{rr}}{\phi_r}+(\tau-n\epsilon)\phi-n=\mu
\phi_r
\end{equation}
Since $\phi_r=P_{rr}=(P_ss)_ss=(P_s+P_{ss}s)s>0$ by (\ref{poscond}),
we can solve $r$ as a function of $\phi$: $r=r(\phi)$. Define
$F(\phi)=\phi_{r}(r(\phi))$, then
$F'(\phi)=\phi_{rr}r'(\phi)=\frac{\phi_{rr}}{\phi_r}$. So the above
equation change into an ODE
\begin{align}
F'(\phi)+d\frac{\epsilon
F(\phi)}{1+\epsilon\phi}+(n-1)\frac{F(\phi)}{\phi}-\mu
F(\phi)=n-(\tau-n\epsilon)\phi=n(1+\epsilon\phi)-\tau\phi\tag*{$(*)$}\label{ode}
\end{align}
\begin{rem}\label{relate}
We will explain how this equation is related to the ODE in
\cite{FIK},(25). In our notation, in \cite{FIK}, $M=\mathbb{P}^{d}$,
$L=O_{\mathbb{P}^d}(-k)$, $n=1$. For the shrinking soliton case,
$d+1-k>0$, $\omega_M=(d+1-k)\omega_{FS}$, $\tau=\frac{d+1}{d+1-k}$,
$\epsilon=\frac{k}{d+1-k}$, $\lambda=\tau-\epsilon=1$. Let
$r=k\tilde{r}$,
$P(r)=\tilde{P}(\tilde{r})-(d+1-k)\tilde{r}=\tilde{P}(\frac{r}{k})-\frac{d+1-k}{k}r$,
$\phi(r)=P_r(r)=\tilde{P}_{\tilde{r}}(\tilde{r})\frac{1}{k}-\frac{d+1-k}{k}=\frac{1}{k}(\tilde{\phi}(\tilde{r})-(d+1-k))$,
$F(\phi)=\phi_r(r(\phi))=\frac{1}{k^2}\tilde{\phi}_{\tilde{r}}(\tilde{r}(\tilde{\phi}))=\frac{1}{k^2}\tilde{F}(\tilde{\phi})$,
$F'_\phi(\phi)=\frac{1}{k}\tilde{F}'_{\tilde{\phi}}(\tilde{\phi})$.
Substitute these expressions into \ref{ode}, then we get the ODE
\[
\tilde{F}'_{\tilde{\phi}}+\left(\frac{d}{\tilde{\phi}}-\frac{\mu}{k}\right)\tilde{F}-((d+1)-\tilde{\phi})=0
\]
So we see this is exactly the ODE in \cite{FIK}, (25). The expanding
soliton case is the similar.
\end{rem}
We can solve \ref{ode} by multiplying the integral factor:
$(1+\epsilon\phi)^{d}\phi^{n-1}e^{-\mu\phi}$:
\begin{equation}\label{phir}
\phi_r=F(\phi)=\nu(1+\epsilon\phi)^{-d}\phi^{1-n}e^{\mu\phi}-(1+\epsilon\phi)^{-d}\phi^{1-n}e^{\mu\phi}\int
h(\phi)e^{-\mu\phi}d\phi
\end{equation}
where
\begin{equation}\label{hpoly}
h(\phi)=\tau(1+\epsilon\phi)^{d}\phi^n-n(1+\epsilon\phi)^{d+1}\phi^{n-1}
\end{equation}
is a polynomial of $\phi$ with degree $d+n$.  Note the identity:
\[
\int
h(\phi)e^{-\mu\phi}d\phi=-\sum_{k=0}^{+\infty}\frac{1}{\mu^{k+1}}h^{(k)}(\phi)e^{-\mu\phi}
\]
Since $h(\phi)$ is a polynomial of degree d+n, the above sum is a
finite sum. So
\begin{equation}\label{Fphi}
F(\phi)=(1+\epsilon\phi)^{-d}\phi^{1-n}\left(\nu
e^{\mu\phi}+\sum_{k=0}^{d+n}\frac{1}{\mu^{k+1}}h^{(k)}(\phi)\right)
\end{equation}
\end{section}

\begin{section}{Boundary condition at zero section}
Since
$\lim_{r\rightarrow-\infty}\phi(r)=\lim_{s\rightarrow0}P_ss=0$, we
have the boundary condition
\begin{equation}\label{bdry1}
\lim_{\phi\rightarrow0}F(\phi)=\lim_{r\rightarrow-\infty}\phi_r=\lim_{s\rightarrow0}(P_ss)_ss=0
\end{equation}
So $\phi^{n-1}(1+\epsilon\phi)^{d}F(\phi)=O(\phi^n)$. Now the $l$-th
term of Taylor expansion of $\phi^{n-1}(1+\epsilon\phi)^{d}F(\phi)$
at $\phi=0$ is
\begin{equation}
\left.(\phi^{n-1}(1+\epsilon\phi)^{d}F(\phi))^{(l)}\right|_{\phi=0}=\nu\mu^l+\sum_{k=0}^{+\infty}\frac{1}{\mu^{k+1}}h^{(k+l)}(0)=\mu^{l}\left(\nu+\sum_{k=l}^{d+n}\frac{1}{\mu^{k+1}}h^{(k)}(0)\right)
\end{equation}
Note that by \eqref{hpoly} $h^{(k)}(\phi)=0$ for $k>d+n$, and
$h^{(k)}(0)=0$ for $k<n-1$. The vanishing of the 0-th(constant) term
in expansion gives the equation:
\begin{equation}\label{numu1}
\nu+\sum_{k=n-1}^{d+n}\frac{1}{\mu^{k+1}}h^{(k)}(0)=0
\end{equation}
Using relation (\ref{numu1}) we see that, when $l<n$,
$\left.(\phi^{n-1}(1+\epsilon\phi)^{d}F(\phi))^{(l)}\right|_{\phi=0}=0$,
and
\[\left.(\phi^{n-1}(1+\epsilon\phi)^{d}F(\phi))^{(n)}\right|_{\phi=0}=\mu^{n}\left(\nu+\sum_{k=n}^{d+n}\frac{1}{\mu^{k+1}}h^{(k)}(0)\right)=-h^{(n-1)}(0)=n!>0\]
So we see that (\ref{bdry1}) and (\ref{numu1}) are equivalent, and
if they are satisfied,
\[
\phi^{n-1}(1+\epsilon\phi)^{d}F(\phi)=\phi^n+O(\phi^{n+1}),\;\mbox{or}\;
F(\phi)=\phi+O(\phi^2)
\]
So $F(\phi)>0$ for $\phi$ near 0.

We can rewrite the relation (\ref{numu1}) more explicitly:
\begin{eqnarray}\label{numu}
\nu&=&\left.
\sum_{k=0}^{+\infty}\frac{1}{\mu^{k+1}}\left[n((1+\epsilon\phi)^{d+1}\phi^{n-1})^{(k)}-\tau((1+\epsilon\phi)^{d}\phi^n)^{(k)}\right]\right|_{\phi=0}\nonumber\\
&=&\sum_{k=n-1}^{d+n}\frac{1}{\mu^{k+1}}\left(n\binom{k}{n-1}(n-1)!\frac{(d+1)!}{(d+n-k)!}\epsilon^{k-n+1}-\tau\binom{k}{n}n!\frac{d!}{(d+n-k)!}\epsilon^{k-n}\right)\nonumber\\
&=&\sum_{k=n-1}^{d+n}C_k\frac{1}{\mu^{k+1}}\epsilon^{k-n}
\end{eqnarray}
Here
\begin{equation}\label{Ck}
C_k=\frac{k!d!}{(k-n+1)!(d+n-k)!}(n\epsilon(d+1)-\tau(k-n+1))
\end{equation}
\[
C_{n-1}=n!\epsilon,\quad C_{d+n}=-(d+n)!(\tau-n\epsilon)
\]
So, when $k$ starts from $n-1$ to $d+n$, $C_k$ change signs from
positive to negative if and only if $\lambda=\tau-n\epsilon>0$. We
need the following simple lemma later.
\begin{lem}\label{rootlem}
Let $P(x)=\sum_{i=0}^la_ix^i-\sum_{j=l+1}^{N}a_jx^j$ be a polynomial
function. Assume $a_i>0$ for $0\le a_i\le N$. Then there exists a
unique root for $P(x)$ on $[0,\infty)$.
\end{lem}
\begin{proof}
First $P(0)=a_0>0$. Since $a_N<0$, when $x$ is large enough
$P(x)<0$. So there exists at least one root on $[0,\infty)$. Assume
there are more than one root, than it's easy to see that $P'(x)$ has
at least two roots on $[0,\infty)$. Note that $P'(x)$ has the same
form as $P(x)$, so $P''(x)$ has at least two roots on $[0,\infty)$.
By induction, $P^{(l)}(x)$ has at least two roots on $[0,\infty)$,
but $P^{(l)}(x)$ has only negative coefficients, so it has no root
at all. This contradiction proves the lemma.
\end{proof}
\end{section}
\begin{section}{Complete noncompact case}
We prove theorem \ref{KRsoliton} in this section.
\begin{subsection}{Condition at infinity}
As $\phi\rightarrow+\infty$,
\begin{equation}\label{infasym}
F(\phi)=\nu(1+\epsilon\phi)^{-d}\phi^{1-n}e^{\mu\phi}+\frac{\tau-n\epsilon}{\mu}\phi+O(1)
\end{equation}

Let $\phi=b_1>0$ be the first positive root for $F(\phi)=0$, then
$F'(b_1)\le0$. By \ref{ode}, $F'(b_1)=n-(\tau-n\epsilon)b_1$. So if
$\lambda=\tau-n\epsilon\le0$, there exits no such $b_1$. If
$\lambda=\tau-n\epsilon>0$, we integrate (\ref{phir}) to get
\begin{equation}\label{rphi}
r=r(\phi)=\int_{\phi_0}^{\phi}\frac{1}{F(u)}du+r(\phi_0)
\end{equation}
then the metric is defined for $-\infty<r<r(b_1)$. 

We can also calculate the length of radial curve extending to
infinity. In a fixed fibre, the radial vector
\[
\frac{\partial}{\partial
r}=\frac{1}{2}|\xi|\frac{\partial}{\partial|\xi|}=\frac{1}{2}\sum_{\alpha=1}^n\xi^\alpha\frac{\partial}{\partial
\xi^\alpha}=\frac{1}{2}V_{rad}
\]
\[
\left|\frac{\partial}{\partial
r}\right|^2=\frac{1}{4}g_\omega(V_{rad},V_{rad})=C(P_ss+P_{ss}s^2)=C\phi_r
\]
The completeness implies that the length of the radial curve
extending to infinity is infinity:
\begin{equation}\label{comp2}
\int_{-\infty}^{r(b_1)}\sqrt{\phi_r}dr=\int_{0}^{b_1}
\sqrt{\phi_r}\phi_r^{-1}d\phi=\int_{0}^{b_1}\phi_r^{-\frac{1}{2}}d\phi=\int_{0}^{b_1}F(\phi)^{-\frac{1}{2}}d\phi=+\infty
\end{equation}
If $0<b_1<+\infty$,(\ref{comp2}) means
$F(\phi)=c(\phi-b_1)^2+O((\phi-b_1)^3)$, i.e. $F'(b_1)=F(b_1)=0$.

But this can't happen: $b_1=\frac{n}{\tau-n\epsilon}$, and
$c=-(\tau-n\epsilon)$. First we have $b_1>0$. Second, $c\ge 0$ since
$F(\phi)>0$ when $\phi<b_1$. But they contradict with each other.

In conclusion, there can't be any finite value positive root for
$F(\phi)$.
\end{subsection}
\begin{subsection}{Existence and asymptotics}
\begin{enumerate}
\item ($\lambda=\tau-n\epsilon>0$) The solution is a shrinking
K\"{a}hler-Ricci soliton. If $\mu<0$, then when $\phi$ becomes
large, the dominant term in $F(\phi)$(\ref{infasym}) is
$\frac{\lambda}{\mu}\phi<0$, so there exists $0<b<+\infty$ such that
$F(b)=0$. But this is excluded by former discusions. So we must have
$\mu>0$.

If $\nu<0$ the dominant term is $\nu
\phi^{1-n}(1+\epsilon\phi)^{-m}e^{\mu\phi}<0$, so there is
$0<b<+\infty$, such that $F(b)=0$. Again, this is impossible. If
$\nu>0$, when $\phi$ becomes large, the dominant term is
$\nu\phi^{1-d-n}e^{\mu\phi}$,
\[
\int_{\phi_0}^{+\infty}F(s)^{-\frac{1}{2}}ds\leq
C\int_{\phi_0}^{+\infty}\nu^{-\frac{1}{2}}\phi^{\frac{d+n-1}{2}}e^{-\frac{\mu}{2}\phi}d\phi<+\infty
\]
This contradicts (\ref{comp2}). So we must have $\nu=0$. This gives
us an equation for $\mu$ via (\ref{numu}). Since when
$\lambda=\tau-n\epsilon>0$, $C_k$ change signs exactly once, by
lemma \ref{rootlem}, there exists a unique $\mu$ such that
$\nu(\mu)=0$ in (\ref{numu}). 

We now verify this $\mu$ guarantees
the positivity of $\phi_r$. Since the dominant term in
(\ref{infasym}) is $\frac{\lambda}{\mu}\phi>0$,
$F(\phi)\stackrel{\phi\rightarrow+\infty}{\longrightarrow}+\infty$.
We have also $F(\phi)>0$ for $\phi$ near 0. If $\phi=b_1>0$ is the
first root and $\phi=b_2>0$ is the last root of $F(\phi)$, then
$b_1\le b_2$, and
\[
F'(b_1)=-\lambda b_1+n\le0,\;\; F(b_2)=-\lambda b_2+n\ge0
\]
So $b_1\ge\frac{n}{\lambda}\ge b_2$, this implies $b_1=b_2$ and
$F'(b_1)=0$. We have ruled out this possibility before. In
conclusion, $F(\phi)>0$ for all $\phi>0$, or equivalently $\phi_r>0$
for all $r$.

Now in \eqref{infasym}, $\phi=\frac{\tau-n\epsilon}{\mu}\phi+O(1)$ as $\phi\rightarrow+\infty$, so by \eqref{rphi}, the maximum value of $r$ defined for the solution is 
\begin{equation}\label{comp1}
r_{max}=\int_0^{+\infty}\frac{1}{F(u)}du=+\infty
\end{equation}

So we already get the soliton on the whole manifold. In the following, we study the limit
of flow as time approaches singularity time.

Define $p=\frac{\lambda}{\mu}=\frac{\tau-n\epsilon}{\mu}$,
\[
r(\phi)-r(\phi_0)=\int_{\phi_0}^\phi\frac{du}{F(u)}=\int_{\phi_0}^\phi\frac{du}{pu}+\int_{\phi_0}^\phi\frac{pu-F(u)}{puF(u)}ds=\frac{1}{p}(\log\phi-\log\phi_0)+G(\phi_0,\phi)
\]
\begin{equation}\label{phi}
\phi(r)=\phi_0e^{-pr(\phi_0)}e^{-G(\phi_0,\phi)}e^{pr}=\phi_0e^{-pr(\phi_0)}e^{-G(\phi_0,\phi(r))}s^{p}
\end{equation}
The holomorphic vector field
$-\frac{V}{2}=\frac{\mu}{2}\sum_\alpha\xi^\alpha\frac{\partial}{\partial\xi^\alpha}$
generates the 1-parameter family of automorphisms:
$\sigma(\tilde{t})\cdot(u,\xi)=(u,e^{\frac{\tilde{t}\mu}{2}}\xi)$.
Let
\begin{equation}\label{tildet}
\tilde{t}(t)=-\frac{1}{\lambda}\log(1-\lambda t)
\end{equation}
\begin{eqnarray*}
\lim_{t\rightarrow\frac{1}{\lambda}}(1-\lambda
t)\sigma(\tilde{t})^*\phi
&=&\phi_0e^{-pr(\phi_0)}e^{-\lim_{\phi\rightarrow+\infty}G(\phi_0,\phi)}\lim_{t\rightarrow\frac{1}{\lambda}}(1-\lambda
t)((1-\lambda t)^{-\frac{\mu}{\lambda}}s)^p\\
&=&\phi_0e^{-pr(\phi_0)}e^{-G(\phi_0,+\infty)}s^p=D_0 s^p
\end{eqnarray*}
So
\[
\lim_{t\rightarrow\frac{1}{\lambda}}(1-\lambda
t)\sigma(\tilde{t})^*\phi_r=\lim_{t\rightarrow\frac{1}{\lambda}}(1-\lambda
t)\sigma(\tilde{t})^*F(\phi)=pD_0s^p
\]
(\ref{omega}) can be rewritten as
\begin{equation}
\omega=(1+\epsilon\phi)\pi^*\omega_{FS}+(\phi|\xi|^{-2}\delta_{\alpha\beta}+(\phi_r-\phi)|\xi|^{-4}\overline{\xi^\alpha}\xi^\beta)\nabla
\xi^\alpha\wedge\overline{\nabla\xi^\beta}
\end{equation}
\begin{eqnarray*}
\lim_{t\rightarrow\frac{1}{\lambda}}(1-\lambda
t)\sigma(\tilde{t})^*\omega&=&D_0\left[s^{p}\epsilon\omega_M
+s^{p}(|\xi|^{-2}\delta_{\alpha\beta}+(p-1)|\xi|^{-4}\overline{\xi^\alpha}\xi^\beta)\nabla\xi^\alpha\wedge\overline{\nabla\xi^\beta}\right]\\
&=&D_0\partial\bar{\partial}\left(\frac{1}{p}s^p\right)
\end{eqnarray*}
\begin{rem}
One sees that as $t\rightarrow \frac{1}{\lambda}$, the flow shrinks
the base (zero section of the vector bundle). In the model case,
$M=\mathbb{P}^{m-1}$, $L=\mathcal{O}(-1)$, the flow contracts the
manifold to the affine cone of the Segre embedding
$\mathbb{P}^{m-1}\times\mathbb{P}^{n-1}\hookrightarrow\mathbb{P}^{mn-1}$.
This is the same phenomenon as that appears for the symplectic
quotients at the beginning of this note.
\end{rem}
\item ($\lambda=\tau-n\epsilon=0$) the solution is a steady K\"{a}hler-Ricci soliton.
\[
F(\phi)=\nu(1+\epsilon\phi)^{-d}\phi^{1-n}e^{\mu\phi}-n(1+\epsilon\phi)^{-d}\phi^{1-n}\sum_{k=0}^{d+n-1}\frac{1}{\mu^{k+1}}\left((1+\epsilon\phi)^{d}\phi^{n-1}\right)^{(k)}
\]
If $\mu>0$, then $v(\mu)>0$ by (\ref{numu}). So the dominant term in
(\ref{infasym}) is $\nu(1+\epsilon\phi)^{-d}\phi^{1-n}e^{\mu\phi}$, again this would contradict 
\eqref{comp2}.

So $\mu<0$ and the dominant term in (\ref{Fphi}) is the constant
term $-\frac{n}{\mu}>0$. As $\phi\rightarrow+\infty$,
\[
F(\phi)=-\frac{n}{\mu}-\frac{n(d+n-1)}{\mu^2}\frac{1}{\phi}+O\left(\frac{1}{\phi^2}\right)=c_1-c_2\frac{1}{\phi}+O\left(\frac{1}{\phi^2}\right)
\]
So by \eqref{rphi}, $r_{max}=+\infty$ and the soliton is defined on the whole manifold. 

For the asymptotic behavior, let
\[
\int\frac{du}{c_1-\frac{c_2}{u}}=\frac{1}{c_1}u+\frac{c_2}{c_1^2}\log(c_1u-c_2)=R(u)
\]
\begin{eqnarray*}
r(\phi)-r(\phi_0)&=&\int_{\phi_0}^{\phi}\frac{du}{F(u)}=\int_{\phi_0}^\phi\frac{du}{c_1-\frac{c_2}{u}}+\int_{\phi_0}^\phi\left(\frac{1}{F(u)}-\frac{1}{c_1-\frac{c_2}{u}}\right)du\\
&=&R(\phi)-R(\phi_0)+G(\phi_0,\phi)
\end{eqnarray*}
Since $c_1>0$ and $c_2>0$ ($\mu<0$, $d>0$, $n\ge1$), $R(u)$ is an
increasing function for $u\gg0$, and has an inverse function denoted
by $R^{-1}$. Let
$\tilde{G}(r)=-G(\phi_0,\phi(r))+R(\phi_0)-r(\phi_0)$, then
$\tilde{G}$ is a bounded smooth function of $r$. We have
\[
\phi(r)=R^{-1}(r+\tilde{G}(r))
\]
The condition (\ref{comp2}) is always satisfied. There is a family
of steady K\"{a}hler-Ricci solitons.
\begin{rem}
If we let $d=0$, then we get expanding solitons on $\mathbb{C}^n$.
\eqref{numu1} becomes $\nu\mu^n=n!$.  The equation becomes
\[
(|\mu|\phi)_r=(-1)^nn!(|\mu|\phi)^{1-n}e^{-|\mu|\phi}+n!\sum_{k=0}^{n-1}\frac{(-1)^{k}}{(n-1-k)!(|\mu|\phi)^k}
\]
In particular, if $n=1$, the equation becomes
\[
\phi_r=F(\phi)=\nu e^{\mu\phi}-\frac{1}{\mu}
\]
\[
\phi(r)=-\frac{1}{\mu}\log(\mu\nu+Ce^{r})=-\frac{1}{\mu}\log(1+C|z|^2),\quad
\phi_r=-\frac{C|z|^2}{\mu(1+C|z|^2)}
\]
\[
\omega=\frac{\phi_r}{|z|^2}dz\wedge d\bar{z}=-\frac{Cdz\wedge
d\bar{z}}{\mu(1+C|z|^2)}\stackrel{w=\sqrt{C}z}{=}\frac{1}{-\mu}\frac{dw\wedge
d\bar{w}}{1+|w|^2}
\]
This is cigar steady soliton.
\end{rem}

\item ($\lambda=\tau-n\epsilon<0$) the solution is an expanding K\"{a}hler-Ricci soliton. By
similar argument, we see that $\mu<0$. The situation is similar to
the shrinking soliton case. Now $t\rightarrow\frac{1}{\lambda}<0$,
or equivalently $\tilde{t}\rightarrow-\infty$ (\ref{tildet}),
\[
\phi(r)=\phi_0e^{-pr(\phi_0)}e^{-G(\phi_0,\phi(r))}s^{p}
\]
\[
\lim_{t\rightarrow\frac{1}{\lambda}}(1-\lambda
t)\sigma(\tilde{t})^*\phi_r=\lim_{t\rightarrow\frac{1}{\lambda}}(1-\lambda
t)\sigma(\tilde{t})^*F(\phi)=pD_0s^p
\]
\[
\lim_{t\rightarrow\frac{1}{\lambda}}(1-\lambda
t)\sigma(\tilde{t})^*\omega=D_0\partial\bar{\partial}\left(\frac{1}{p}s^p\right)
\]
The condition (\ref{comp2}) is always satisfied. So there is a
family of expanding K\"{a}hler-Ricci solitons.
\end{enumerate}
\begin{rem}
One can apply the same argument in \cite{FIK} to get the
Gromov-Hausdorff convergence and continuation of flow through
singularity time.
\end{rem}
\end{subsection}
\end{section}
\begin{section}{Compact shrinking soliton}
We prove theorem \ref{cptshrink} in this section. First we show the
considered manifold is Fano. For some results on Fano manifolds with
a structure of projective space bundle, see \cite{SW}.
\begin{lem}
If $\lambda=\tau-n\epsilon>0$, then $\mathbb{P}(\mathbb{C}\oplus
L^{\oplus n})$ is Fano.
\end{lem}
\begin{proof}
Let $E=\mathbb{C}\oplus L^{\oplus n}$ and $X=\mathbb{P}(E)$. We have
the formula for anti-canonical bundle:
\[
K_{X}^{-1}=(n+1)\mathcal{O}(1)+\pi^*(K_M^{-1}+L^{\otimes n})
\]
$\mathcal{O}(1)$ is the relative hyperplane bundle. Since
$c_1(L^{-1})=\epsilon[\omega_M]\ge0$, one can prove $\mathcal{O}(1)$
is nef on $X$ \cite{Laz}. $c_1(K_M^{-1}+L^{\otimes
n})=(\tau-n\epsilon)[\omega_M]>0$, so $K_M^{-1}+L^{\otimes n}$ is an
ample line bundle on $M$. So $\mathcal{O}(1)$ and
$\pi^*(K_M^{-1}+L^{\otimes n})$ are different rays of the cone of
numerically effective divisors in
$Pic(\mathbb{P}(E))=\mathbb{Z}Pic(M)+\mathbb{Z}\mathcal{O}(1)$. So
$K_{X}^{-1}$ is ample, i.e. $X$ is Fano.
\end{proof}

The construction of shrinking soliton was developed for the n=1
case, see \cite{Ca2}, \cite{Koi}, \cite{Cao}, \cite{FIK}. We will
give a simple direct argument under our setting. Note here we will
use Tian-Zhu's theory \cite{TZ} to get the uniqueness of
K\"{a}hler-Ricci soliton.

First we need to know the expression for the metric near infinity.
By change of coordinate
\[
[1,\xi_1,\xi_2,\cdots,\xi_n]=[\eta,1,u_2,\cdots, u_n]
\]
So the coordinate change is given by
\[
\xi_1=\frac{1}{\eta}, \xi_2=\frac{u_2}{\eta}, \cdots,
\xi_n=\frac{u_n}{\eta}\quad \Longleftrightarrow\quad
\eta=\frac{1}{\xi_1}, u_2=\frac{\xi_2}{\xi_1}, \cdots,
u_n=\frac{\xi_n}{\xi_1}
\]
Since
\[
\xi_1\frac{\partial}{\partial\xi_1}=-\eta\frac{\partial}{\partial
\eta}-\sum_{i=2}^nu_\alpha\frac{\partial}{\partial u_\alpha},\quad
\xi_\alpha\frac{\partial}{\partial\xi_\alpha}=u_\alpha\frac{\partial}{\partial
u_\alpha}
\]
So the radial vector
$\sum_{i=\alpha}^n\xi_\alpha\frac{\partial}{\partial\xi_\alpha}=-\eta\frac{\partial}{\partial\eta}$
is a holomorphic vector field on $\mathbb{P}(\mathbb{C}\oplus
L^{\oplus n})$. The dual 1-forms transform into
\[
\nabla\xi^1=-\frac{1}{\eta^2}(d\eta-\eta a^{-1}\partial
a)=-\frac{1}{\eta^2}\omega_0,\quad
\nabla\xi^\alpha=\frac{du_\alpha}{\eta}-\frac{u_\alpha}{\eta^2}(d\eta-\eta
a^{-1}\partial
a)=\frac{du_\alpha}{\eta}-\frac{u_\alpha}{\eta^2}\omega_0
\]
Note that the dual basis for the basis $\{dz_i, \omega_0,
du_\alpha\}$ is
\[
\nabla_{z_i}=\frac{\partial}{\partial z_i}+a^{-1}\frac{\partial
a}{\partial z_i}\eta\frac{\partial}{\partial\eta},\quad
\frac{\partial}{\partial \eta},\quad \frac{\partial}{\partial
u_\alpha}
\]
So making coordinate change,
\begin{eqnarray}\label{omegainf}
\omega&=&(1+\epsilon P_ss)\omega_M+\sum_{\alpha=2}^n\sum_{\beta=2}^na\left(P_s\delta_{\alpha\beta}+P_{ss}a\frac{\bar{u}_\alpha
u_\beta}{|\eta|^2}\right)\left(\frac{1}{\eta}du_\alpha-\frac{u_\alpha}{\eta^2}\omega_0\right)\wedge\left(\frac{1}{\bar{\eta}}d\bar{u}_\beta-\frac{\bar{u}_\beta}{\bar{\eta}^2}\overline{\omega_0}\right)\nonumber\\
&&-\sum_{\beta=2}^n P_{ss}a^2
\frac{u_\beta}{|\eta|^2}\frac{\omega_0}{\eta^2}\wedge
\left(\frac{1}{\bar{\eta}}d\bar{u}_\beta-\frac{\bar{u}_\beta}{\bar{\eta}^2}\overline{\omega_0}\right)
-\sum_{\alpha=2}^nP_{ss}a^2\frac{\bar{u}_\alpha}{|\eta|^2}\left(\frac{1}{\eta}du_\alpha-\frac{u_\alpha}{\eta^2}\omega_0\right)\wedge\frac{\overline{\omega_0}}{\bar{\eta}^2}\nonumber\\
&&+
a(P_s+P_{ss}a\frac{1}{|\eta|^2})\frac{\omega_0\wedge\overline{\omega_0}}{|\eta|^4}\nonumber\\
&=&(1+\epsilon P_ss)\omega_M+\sum_{\alpha=2}^n\sum_{\beta=2}^n(P_s\delta_{\alpha\beta}+P_{ss}s\frac{\bar{u}_\alpha
u_\beta}{1+|u|^2})\frac{s}{1+|u|^2}du_\alpha\wedge
d\bar{u}_\beta\nonumber\\
&&-\sum_{\alpha=2}^n\frac{\bar{u}_\alpha\eta}{a(1+|u|^2)^2}(P_s+P_{ss}s)s^2
du_\alpha\wedge\overline{\omega_0}-\sum_{\beta=2}^n\frac{u_\beta\bar{\eta}}{a(1+|u|^2)^2}(P_s+P_{ss}s)s^2\omega_0\wedge d\bar{u}_\beta\\
&&+\frac{1}{a(1+|u|^2)}(P_s+P_{ss}s)s^2\omega_0\wedge\overline{\omega_0}\nonumber
\end{eqnarray}
In the above calculation, we used many times the relation
$s=a|\xi|^2=\frac{a(1+|u|^2)}{|\eta|^2}$.
\begin{lem}
The closing condition for compact shrinking soliton is: there exists
a $b_1>0$, such that
\begin{equation}\label{cptcond}
F(b_1)=0, F'(b_1)=-1
\end{equation}
\end{lem}
\begin{proof}
Define $\tilde{s}=s^{-1}=\frac{|\eta|^2}{a(1+|u|^2)}$. Under the
condition \eqref{cptcond}, then near $b_1$,
$\phi_r=F(\phi)=-(\phi-b_1)+O((\phi-b_1)^2)$. So up to the main term,
$\phi-b_1\sim -C_0e^{-r}=-C_0\frac{1}{s}=-C_0\tilde{s}$ for some
$C_0>0$, $(P_ss)_ss^2=\phi_ss^2\sim C_0$,
$P_{ss}s^2=(P_ss)_ss-P_ss=\phi_ss-\phi\sim
-b_1+\frac{2C_0}{s}=-b_1+2C_0\tilde{s}$, . So we first see that the
coefficients in \eqref{omegainf} are smooth near infinity divisor
defined by $\eta=0$ (or equivalently $\tilde{s}=0$). We only need to
show $\omega$ is positive definite everywhere. In fact, we only need
to check when $\tilde{s}=0$. When $\tilde{s}=0$, we have
\begin{equation}\label{infexp}
\omega=(1+\epsilon b_1)\omega_M+\sum_{\alpha=2}^n\sum_{\beta=2}^n(b_1\delta_{\alpha\beta}-b_1\frac{\bar{u}_\alpha
u_\beta}{1+|u|^2})\frac{1}{1+|u|^2}du_\alpha\wedge
d\bar{u}_\beta+\frac{C_0}{a(1+|u|^2)}\omega_0\wedge\overline{\omega_0}
\end{equation}
So $\omega$ is positive definite. So it defines a smooth K\"{a}hler
metric on the projective compactification.
\end{proof}
By \ref{ode}, this condition determines
\begin{equation}\label{b1}
b_1=\frac{n+1}{\tau-n\epsilon}=\frac{n+1}{\lambda}
\end{equation}
In fact, $b_1$ is a cohomological constant. Indeed.
\[
[\omega]=\frac{1}{\lambda}c_1(X)=[\omega_M]+\frac{n+1}{\lambda}\mathcal{O}_X(1)
\]
Note that $D_\infty=M\times \mathbb{P}^{n}$, and 
\[
[\omega]|_{D_\infty}=[\omega_M]+\frac{n+1}{\lambda}(-L+\mathcal{O}_{\mathbb{P}^{n}}(1))
\]
On the otherhand, by \eqref{infexp}, we see that
\[
[\omega]|_{D_\infty}=(1+\epsilon b_1)[\omega_M]+C_0[\mathcal{O}_{\mathbb{P}^{n}}(1)]
\]
Comparing the above expression and using $c_1(L)=-\epsilon[\omega_M]$, we get \eqref{b1}.

Since $F(0)=0$, this condition is equivalent to
\begin{equation}\label{cpteq}
0=F(b_1)-F(0)=\int_0^{b_1}h(\phi)e^{-\mu\phi}d\phi=T(\mu)
\end{equation}
\begin{eqnarray*}
T(0)&=&\int_0^{b_1}h(\phi)d\phi=\int_0^{b_1}(1+\epsilon\phi)^d\phi^{n-1}((\tau-n\epsilon)\phi-1)\\
&=&\int_0^{b_1}\sum_{k=0}^d\binom{d}{k}\epsilon^k(\lambda\phi^{k+n}-\phi^{k+n-1})d\phi\\
&=&\sum_{k=0}^d\binom{d}{k}\epsilon^k
b_1^{k+n}\left(\frac{\lambda b_1}{k+n+1}-\frac{n}{k+n}\right)\\
&=&\sum_{k=0}^d\binom{d}{k}\epsilon^kb_1^{k+n}\frac{k}{(k+n+1)(k+n)}>0
\end{eqnarray*}
On the other hand,
\[
T(\mu)=\frac{1}{\mu^{d+n+1}}\sum_{k=0}^{d+n}\mu^{d+n-k}(h^{(k)}(0)-h^{(k)}(b_1)e^{-\mu
b_1})
\]
Since $h^{(i)}(0)=0$ for $0\le i\le n-2$ and $h^{(n-1)}(0)=-n!<0$, and
$\lim_{\mu\rightarrow+\infty}e^{-\mu b_1}=0$. It's easy to see that
$T(\mu)<0$ for $\mu$ sufficiently large. So there is a zero point
for $T(\mu)$ on $(0,\infty)$. The uniqueness is difficult to see
directly, but because different solutions of \eqref{cpteq} would
give proportional vector fields and hence proportional potential
functions, by using Tian-Zhu's invariant \cite{TZ}, we indeed have
the uniqueness.
\end{section}

Department of Mathematics, Princeton University, Princeton, NJ
08544, USA

E-mail address: chil@math.princeton.edu
\end{document}